\documentclass[journal]{IEEEtran}

\usepackage{graphicx} 
\usepackage{times} 
\usepackage{comment}
\usepackage{xcolor}
\usepackage{amsmath} 
\usepackage{amssymb}  
\usepackage{mathtools}
\usepackage{amsthm}
\newtheorem{thm}{Theorem}[]           
\newtheorem{lem}[]{Lemma}             
\newtheorem{prop}[]{Proposition}      

\newtheorem{assum}[]{Assumption}

\title{\LARGE \bf
 \mbox{Extremum Seeking Boundary Control for Euler-Bernoulli Beam PDEs}}

\author{Paulo H. F. Biazetto, Gustavo A. de Andrade, Tiago Roux Oliveira and Miroslav Krstic
\thanks{{This work was partially supported by CAPES under grants 88887.629803/2021-00  and 88881.878833/2023-01 (SticAmSud). The authors thank the Brazilian funding agencies CNPq and FAPERJ for the financial support.}}
\thanks{Paulo H. F. Biazetto is with the
Post-graduate Program in Automation and Systems Engineering, Federal University of Santa Catarina, 88040-370, Florian\'{o}polis, Brazil (e-mail:  paulo.biazetto@posgrad.ufsc.br)}%
\thanks{Gustavo A. de Andrade
is with the Department of Automation and Systems Engineering, Federal University of Santa Catarina, 88040-370, Florian\'{o}polis, Brazil. (e-mail: gustavo.artur@ufsc.br).}%
\thanks{Tiago Roux Oliveira is with the Department of Electronics and Telecommunication Engineering, State University of Rio de Janeiro, Rio de Janeiro 20550-900, Brazil (e-mail: tiagoroux@uerj.br).}
\thanks{M. Krstic is with the Department of Mechanical and Aerospace Engineering, University of California, San Diego, CA 92093-0411, USA (e-mail: krstic@ucsd.edu)}
}

\begin{document}

\maketitle

\begin{abstract}
This paper presents the design and analysis of an extremum seeking (ES) controller for scalar static maps in the context of infinite-dimensional dynamics governed by the 1D Euler-Bernoulli (EB) beam Partial Differential Equation (PDE). The beam is actuated at one end (using position and moment actuators). The map's input is the displacement at the beam's uncontrolled end, which is subject to a sliding boundary condition. Notably, ES for this class of PDEs remains unexplored in the existing literature. To compensate for PDE actuation dynamics, we employ a boundary control law via a backstepping transformation and averaging-based estimates for the gradient and Hessian of the static map to be optimized. This compensation controller leverages a Schrödinger equation representation of the EB beam and adapts existing backstepping designs to stabilize the beam. Using the semigroup and averaging theory in infinite dimensions, we prove local exponential convergence to a small neighborhood of the unknown optimal point. Finally, simulations illustrate the effectiveness of the design in optimizing the unknown static map.
\end{abstract}

\section{Introduction}

The Euler–Bernoulli (EB) beam equation can be applied to delineate a lot of flexible mechanical systems such as robotic manipulators \cite{Liu2017}; moving strips \cite{Choi2004}; flexible marine risers \cite{Do2009}; and flexible wings \cite{He2017}. For the past few years, the dynamics and the control method design for flexible systems built on the partial differential equation (PDE) theory have been extensively studied. For instance, a boundary control scheme is designed for a two-dimensional variable-length crane system under external disturbances and constraints to reduce the coupled vibrations in \cite{He2017}. An active control scheme is proposed in \cite{Zhang2016} to suppress a flexible string, in which a novel ‘disturbance-like’ term is designed to deal with the input backlash. It can be proven that the proposed control can prevent the constraint violation. In \cite{Jin2015}, a boundary controller is proposed for an EB beam with external disturbance when PDEs represent the dynamics. Many flexible systems are governed by coupled ordinary differential equations (ODEs) and PDEs. This is illustrated in \cite{He2016}, where an integral barrier Lyapunov function is employed to design cooperative control laws for a gantry crane system whose tension is additionally constrained and described by a hybrid PDE-ODE system. Although great strides in the control of flexible mechanical systems have been made, studies about extremum seeking (ES) for this class of PDEs remain unexplored in the existing literature.

Extremum Seeking is a non-model-based approach in the field of adaptive control that searches in real-time the extremum point of a performance index of a system. This method has received great attention in the control community by facing control problems when the plant has imperfections in its model or uncertainties \cite{Krsti2000StabilityOE}.

In the context of ES control schemes applied to PDEs, the first result was published in \cite{Oliveira2017ExtremumSF}, where the design and analysis of multivariable static maps subject to arbitrarily long time delays were addressed. The delays pointed out by the authors can be modeled as first-order hyperbolic transport PDEs \cite{krstic:2009}. This idea has enabled the development of extensions to other classes of PDEs \cite{book2022}.

In this paper, we explore the ES design for the EB beam PDE with actuation at one end through position and moment actuators. The system's output is the displacement at the uncontrolled end, which is subject to a sliding boundary condition. Our method is based on the well-known representation of the Euler-Bernoulli beam model through the Schrödinger equation \cite{Ren2013}. The theoretical results demonstrate that the local exponential stability of the closed-loop average system is ensured and that convergence to a small neighborhood of the extremum is achieved. Finally, we present simulations to illustrate the effectiveness of the method.

The paper is organized as follows. Section \ref{section:problem} introduces the EB beam model and the corresponding control objectives with ES. In Section \ref{section:extremum_seeking_design}, we present the proposed ES control design. We begin by designing the demodulation and additive probing signals. Next, we derive the error dynamics and design a compensator using a backstepping methodology. The closed-loop stability and asymptotic convergence to the extremum are analyzed in Section \ref{section:stability}. Section \ref{section:simulation} illustrates the control design through simulations. Finally, Section \ref{section:conclusion} brings the concluding remarks and discusses possible extensions of the results. 

\section{Problem Formulation}\label{section:problem}
\subsection{Euler-Bernoulli Beam Mathematical Model}

We consider a flexible beam with a sliding boundary at one end and free at the other end. Without loss of generality, we assume that the beam length, mass density, and flexural rigidity are unitary. The equations are given as
\begin{align}
    & u_{tt}(t,x) + u_{xxxx}(t,x) = 0,\label{eq:euler_bernoulli_pde}\\
    & u_x(t,0) = u_{xxx}(t,0) = 0,\label{eq:euler_bernoulli_bc0}\\
    & u(t,1) = \theta_{1}(t), \quad u_{xx}(t,1) = \theta_{2}(t), \label{eq:euler_bernoulli_bc1}
\end{align}
where $x\in[0,1]$ is the space, $t\in[0,+\infty)$ is the time, $u$ is the displacement of the beam, and $\theta_{1}$ and $\theta_{2}$ are control inputs (position and moment actuation, respectively).

\subsection{Control Problem}
The goal of the ES method is to optimize an unknown static map $y = Q(\Theta)$ through real-time optimization, where $y^{*}$ and $\Theta^{*}$ denote the optimal unknown output and optimizer, respectively, while $y$ represents the measurable output, and $\theta_{1}$ and $\theta_{2}$ are the inputs.  

In this work, the input of the map corresponds to the displacement at the uncontrolled end of the beam, which is subject to a sliding boundary condition (see Equation \eqref{eq:euler_bernoulli_bc0}). Thus, we define  
\begin{align}
    \Theta(t) = u(t,0).
\end{align}

\begin{assum}\label{assump:quadraticmap}
The unknown nonlinear map is assumed to be locally quadratic, i.e., 
\begin{equation}
    Q(\Theta(t)) = y^{*} + \frac{H}{2}(\Theta(t) - \Theta^{*})^{2},
    \label{extrachapter.eq:static_map}
\end{equation}
where $y^{*}, \Theta^{*} \in \mathbb{R}$, and $H < 0$ represents the Hessian.  
\end{assum}

Assumption \ref{assump:quadraticmap} is reasonable since every nonlinear function in $C^{2}(\mathbb{R})$ can be approximated as a quadratic function in the neighborhood of its extremum. Therefore, all stability results derived in this section hold at least locally.

Thus, the output of the static map is given by  
\begin{equation}
    y(t) = y^{*} + \frac{H}{2}(\Theta(t) - \Theta^{*})^{2}.
    \label{extrachapter.eq:final_output_static_map}
\end{equation}

\section{Extremum Seeking Boundary Control Design} \label{section:extremum_seeking_design}
\subsection{Demodulation Signals}
The demodulation signal $N(t)$ which is used to estimate the Hessian of the static map by multiplying it with the output $y(t)$ of the static map is defined in \cite{Ghaffari2011MultivariableNE} as
\begin{align}
\hat{H}(t) = N(t)y(t),\quad &\text{with} \quad N(t) = -\dfrac{8}{a^{2}}\cos{(2\omega t)},
    \label{eq:hessian}
\intertext{whereas the signal $M(t)$ is used to estimate the gradient of the static map as follows:}
    G(t) = M(t)y(t),\quad &\text{with} \quad M(t) = \frac{2}{a}\sin{(\omega t)}. \label{eq:gradient_estimate}
\end{align}

\subsection{Additive Probing Signal}
The perturbation $S(t)$ is adapted from the basic ES to the case of PDE actuation dynamic \cite{Oliveira2017ExtremumSF,book2022}. The trajectory generation problem, as in \cite{Krsti2008BoundaryCO}, is described as follows:
\begin{align}
    R_{tt}(t,x) + R_{xxxx}(t,x)&=0, \label{eq:trajectoryPDE}\\
    R_{x}(t,0)=R_{xxx}(t,0)&=0,\label{eq:trajectorybc0}\\
    R(t,0)&=a\sin(\omega t),\label{eq:trajectorybc03}\\
    S_{1}&\coloneqq R(t,1),\label{eq:perturbation1}\\
    S_{2}&\coloneqq R_{xx}(t,1).\label{eq:perturbation2}
\end{align}

The explicit solution of \eqref{eq:trajectoryPDE}-\eqref{eq:perturbation2} is shown in the next Lemma.
\begin{lem}
    The solution of problem \eqref{eq:trajectoryPDE}-\eqref{eq:trajectorybc03} is given by
    \begin{align*}
        R(t,x) = \frac{1}{2}\left[ \cosh (\sqrt{\omega}x)+\cos (\sqrt{\omega}x)\right]a\sin(\omega t).
    \end{align*}

    Additionally,
    \begin{align}
        S_{1}(t) &= \frac{1}{2}\left[\cosh(\sqrt{\omega}) + \cos (\sqrt{\omega})\right]a\sin(\omega t),\label{eq:perturbationS1}\\
        S_{2}(t) &= \frac{\omega}{2}\left[ \cosh(\sqrt{\omega}) - \cos (\sqrt{\omega})\right]a\sin(\omega t).\label{eq:perturbationS2}
    \end{align}
\end{lem}

\begin{proof}
    We postulate the full-state reference trajectory in the form
    \begin{align}
        R(t,x)=\sum_{k=0}^{\infty}a_{k}(t)\frac{x^{k}}{k!}.\label{eq:series_trajectory_postulation}
    \end{align}
    
    Substituting \eqref{eq:series_trajectory_postulation} into \eqref{eq:trajectoryPDE}-\eqref{eq:trajectorybc03}, it follows that
    \begin{align*}
        &a_{0} = a\sin(\omega t) = a\mbox{Im}\{\mathrm{e}^{i\omega t}\},\\
        &a_{1} = a_{2}=a_{3}=0, \\
        &a_{i+4} = -\ddot{a}_{i}.
    \end{align*}
    
    Therefore,
    \begin{align*}
        &a_{4k} = (-1)^{k}a_{0}^{2k} = \omega^{2k}a\sin(\omega t),\\
        &a_{4k+1} = a_{4k+2} = a_{4k+3}=0.
    \end{align*}
    
    The reference trajectory then becomes
    \begin{align}
        R(t,x) &= \sum_{k=0}^{\infty}\omega^{2k}\frac{x^{4k}}{(4k)!}a\sin(\omega t), \nonumber\\
        &= \frac{1}{2}\left[ \cosh (\sqrt{\omega}x)+\cos (\sqrt{\omega}x)\right]a\sin(\omega t).\label{eq:solutionR}
    \end{align}
    
    Finally, substituting \eqref{eq:solutionR} into \eqref{eq:perturbation1}-\eqref{eq:perturbation2}, respectively, we obtain \eqref{eq:perturbationS1}-\eqref{eq:perturbationS2}. The proof is complete.
\end{proof}

\subsection{Estimation Errors and PDE-Error Dynamics}
Since our objective is to determine $\Theta^{*}$, which corresponds to the optimal unknown actuators $\theta_{1}(t)$ and $\theta_{2}(t)$, we introduce the following estimation errors:  
\begin{align}
    \hat{\theta}_{1}(t) &\coloneqq \theta_{1}(t) - S_{1}(t),& \hat{\theta}_{2}(t) &\coloneqq \theta_{2}(t) - S_{2}(t), \label{eq:error_sigma} \\
    \hat{\Theta}(t) &\coloneqq \Theta(t) - a\sin(t). \label{eq:error_theta1} 
\intertext{\indent Furthermore, we define the estimation errors in both the input and propagated input variables as}
    \tilde{\theta}_{1}(t) &\coloneqq \hat{\theta}_{1}(t) - \Theta^{*}, & \tilde{\theta}_{2}(t) &\coloneqq \hat{\theta}_{2}(t) - 0,
    \label{eq:error_thetatilde}\\
    \vartheta(t) &\coloneqq \hat{\Theta}(t) - \Theta^{*}.\label{eq:error_vartheta}
\end{align}

Next, we define  
\begin{equation}
    \alpha(t,x) = u(t,x) - R(t,x) - \Theta^{*}. \label{eq:errorstate}
\end{equation}

Differentiating \eqref{eq:errorstate} with respect to time and substituting \eqref{eq:euler_bernoulli_pde} and \eqref{eq:trajectoryPDE}, we obtain the following error dynamics:  
\begin{align}
    \alpha_{tt}+\alpha_{xxxx} = 0.\label{eq:errorPDE}
\end{align}

By differentiating \eqref{eq:errorstate} with respect to space once and three times, respectively, evaluating at $x=0$, and applying boundary conditions \eqref{eq:euler_bernoulli_bc0} and \eqref{eq:trajectorybc0}, we have  
\begin{align}
    \alpha_{x}(t,0) = \alpha_{xxx}(t,0) = 0.
\end{align}

Additionally, from \eqref{eq:error_vartheta} and the definition in \eqref{eq:errorstate}, we obtain  
\begin{align}
    \vartheta = \alpha(t,0). \label{eq:erroTheta}
\end{align}

Similarly, evaluating \eqref{eq:errorPDE} at $x=1$, differentiating \eqref{eq:errorPDE} twice with respect to $x$, and substituting the boundary conditions \eqref{eq:euler_bernoulli_bc1} and the estimations \eqref{eq:error_sigma} and \eqref{eq:error_thetatilde}, we obtain  
\begin{align}
    \alpha(t,1) = \tilde{\theta}_{1}(t), \qquad \alpha_{xx}(t,1) = \tilde{\theta}_{2}.\label{eq:errorbc1}
\end{align}

Taking the time derivative of \eqref{eq:errorPDE}-\eqref{eq:errorbc1} and defining $U_{1}(t)\coloneqq\dot{\tilde{\theta}}_{1}$ and $U_{2}(t)\coloneqq \dot{\tilde{\theta}}_{2}$, the so-called propagated error dynamics can be expressed as  
\begin{align}
    &\dot{\vartheta}(t) = \beta(t,0),\label{eq:ODE1errordynamics}\\
    &\beta_{tt}(t,x)+\beta_{xxxx}(t,x) = 0,\label{eq:PDEerrordynamics}\\
    &\beta_{x}(t,0)=\beta_{xxx}(t,0) = 0,\label{eq:BC0errordynamics}\\
    &\beta(t,1) = U_{1}(t),\quad \beta_{xx}(t,1) = U_{2}(t),\label{eq:BCUerrordynamics}
\end{align}
where $\beta(t,x) \coloneqq \alpha_{t}(t,x)$.

Adapting the proposed scheme in \cite{book2022} and combining \eqref{eq:euler_bernoulli_pde}-\eqref{extrachapter.eq:final_output_static_map}, the closed-loop ES with actuation dynamics governed by the EB beam PDE is illustrated in Figure \ref{fig:euler}.

\begin{figure*}
    \centering
    \includegraphics[width=0.78\linewidth]{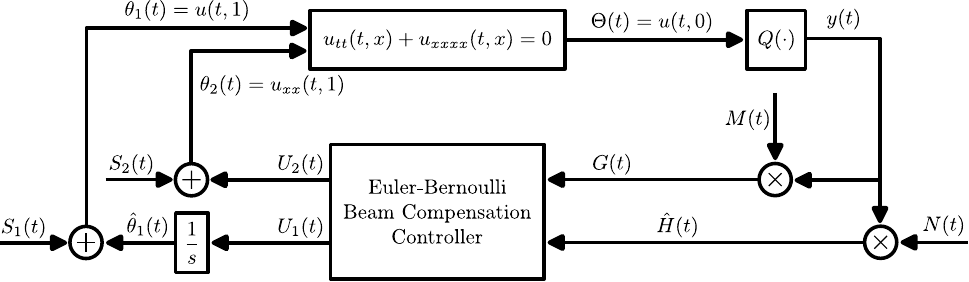}
    \caption{Block diagram of the ES control loop applied to the EB beam problem.}
    \label{fig:euler}
\end{figure*}

\subsection{Euler-Bernoulli Beam Compensation via Backstepping Boundary Control}
To compensate for PDE actuation dynamics, we employ a boundary control law via a backstepping transformation and averaging-based estimates for the gradient and Hessian of the static map to be optimized. This compensation controller leverages a Schrödinger equation representation of the EB beam and extends existing backstepping designs to stabilize the beam. 

As a first step in our design, we transform \eqref{eq:ODE1errordynamics}-\eqref{eq:BCUerrordynamics} into a coupled ODE-Schrödinger system. To achieve this, we introduce the following transformation:
\begin{align}
    v(t,x) = \beta_t(t,x) - i \beta_{xx}(t,x). \label{eq:transformation_schrodinger}
\end{align}

Differentiating \eqref{eq:transformation_schrodinger} with respect to time and twice with respect to space, and substituting \eqref{eq:PDEerrordynamics}, it follows that transformation \eqref{eq:transformation_schrodinger} satisfies the following Schrodinger equation:
\begin{align}
    v_t(t,x) = -i v_{xx}(t,x). \label{eq:schrodingerPDE}
\end{align}

Evaluating \eqref{eq:transformation_schrodinger} at $x = 1$, we obtain the following boundary condition for \eqref{eq:schrodingerPDE}:  
\begin{align}
    v(t,1) = U(t), \label{eq:control_schrodinger}
\end{align}
where $U(t) = \dot{U}_{1}(t) - iU_{2}(t)$.  

The second boundary condition for \eqref{eq:schrodingerPDE} is given by  
\begin{align}
    v_{x}(t,0) = 0,
\end{align}
which follows by differentiating \eqref{eq:schrodingerPDE} with respect to $x$, evaluating the resulting expression at $x = 0$, and substituting \eqref{eq:BC0errordynamics}.  

Finally, we define
\begin{align}
    \dot{\vartheta}_{s}(t) = v(t,0).\label{eq:varthetadynamics}
\end{align}

Note that from \eqref{eq:ODE1errordynamics}, \eqref{eq:transformation_schrodinger} and \eqref{eq:varthetadynamics}  we have $\ddot{\vartheta}(t) = \mbox{Re}\{\dot{\vartheta}_{s}(t)\}=\mbox{Re}\{v(t,0)\}$. Therefore, $\dot{\vartheta}(t) = \mbox{Re}\{\vartheta_{s}(t)\}$

With this formulation, we can now design a backstepping stabilization strategy for the system \eqref{eq:schrodingerPDE}-\eqref{eq:varthetadynamics} and apply it to \eqref{eq:ODE1errordynamics}-\eqref{eq:BCUerrordynamics}.  

\subsubsection{Target System}
We want to map the system \eqref{eq:ODE1errordynamics}-\eqref{eq:BCUerrordynamics} into the following exponentially stable ODE-PDE:
\begin{align}
    \dot{\vartheta}_{s}(t) &= -\overline{K}\vartheta_{s}(t) + w(t,0),\label{eq:target_ode} \\
    w_t(t,x) &= -i w_{xx}(t,x) - c w(t,x), \label{eq:target_pde}\\
    w_x(t,0) &= w(t,1) = 0.\label{eq:target_bc}
\end{align}
\noindent where $\overline{K},c$ are arbitrary pre-defined decay rate.

In order to establish the exponential stability of \eqref{eq:target_ode}-\eqref{eq:target_bc}, let us define the state space $\mathcal{H}=\mathbb{C}\times L^{2}(0,1)$, with the inner product induced norm $$\|(X,f)\|_{\mathcal{H}}=\left(|X|^{2}+\int_{0}^{1}f^{2}(x)dx\right)^{1/2},$$ and define the operator of the system \eqref{eq:target_ode}-\eqref{eq:target_bc} by
\begin{align}
    \mathcal{A}_{w}(X,\; f) = \left(-\overline{K}X + f(0),\;-if''-cf\right),\label{eq:operator_schrodinger}
    \end{align}
 \noindent $\forall (X,\; f)\in D(\mathcal{A}_{w})$, and
    \begin{align}
    \hspace{-0.2cm}D(\mathcal{A}_{w}) =& \left\{ (X,\; f)\in \mathbb{C}\times H^{2}(0,1)| \;f(1) = f'(0)=0 \right\}.\label{eq:domain_operator_schrodinger}
\end{align}

Then, \eqref{eq:target_ode}-\eqref{eq:target_bc} can be written as the following evolution equation in $\mathcal{H}$:
\begin{align}
    \frac{d Y_{w}(t)}{dt} &= \mathcal{A}_{w}Y_{w}(t),\\
    Y_{w}(0) &= Y_{w0},
\end{align}
\noindent where $Y_{w} = (X,\,f)$.

With these definitions in hand, we have the following result \cite{Ren2013}.
\begin{lem}
    Let $\mathcal{A}_{w}$ be defined by \eqref{eq:operator_schrodinger}-\eqref{eq:domain_operator_schrodinger}. Then
    \begin{itemize}
        \item $\mathcal{A}_{w}^{-1}$ exists and is compact on $\mathcal{H}$ and hence the spectrum of $\mathcal{A}_{w}$ consists of isolated eingenvalues of finitely algebraic multiplicity only, which are given by
        \begin{align*}
            \sigma_{0} &= -\overline{K}, &
            \sigma_{m} &= -c +  i\left(m+\frac{1}{2} \right)^{2}\pi^{2},\quad m\in\mathbb{N}.
        \end{align*}
        \item There is a sequence of eigenfunctions of $\mathcal{A}_{w}$ which forms a Riesz basis for $\mathcal{H}$.
        \item $\mathcal{A}_{w}$ generates an exponentially stable C$_{0}$-semigroup $\mathrm{e}^{\mathcal{A}_{w} t}$ in the sense $$\|\mathrm{e}^{\mathcal{A}_{w} t}\|_{\mathcal{H}}\leq M_{1}\mathrm{e}^{-ct},$$ where $M_{1}>0$.
    \end{itemize}
\end{lem}
As highlighted in \cite{Ren2013}, applying a single-step backstepping transformation is challenging due to the complexity of the associated kernels. To overcome this difficulty, a two-step design approach is adopted.

\textbf{First backstepping transformation:} We consider the ODE-PDE \eqref{eq:ODE1errordynamics}-\eqref{eq:BCUerrordynamics} and use the backstepping transformation 
\begin{align}
    z(t,x) = v(t,x) - \int_{0}^{x} q(x,y)v(t,y)dy- \gamma(x)\vartheta_{s}(t),\label{eq:backsteppingtransformation1}
\end{align}
\noindent to transform the original system  \eqref{eq:ODE1errordynamics}-\eqref{eq:BCUerrordynamics} into the target system
\begin{align}
    \dot{\vartheta}_{s}(t) &= -\overline{K}\vartheta_{s}(t) + w(t,0),\label{eq:intermediateODE} \\
    z_{t}(t,x) &= -i z_{xx}(t,x), \label{eq:intermediatePDE}\\
    z_{x}(t,0) &= 0, \quad z(t,1) = W(t).\label{eq:intermediateBC1}
\end{align}
\noindent where $W\in\mathbb{C}$.

The kernels $q$ and $\gamma$ can be shown as
\begin{align}
    &\gamma''(x) = 0,\label{eq:odekernel_intermediate}\\
    &\gamma'(0) = 0,\quad \gamma(0) = -\overline{K}\label{eq:odekernelbc2_intermediate},\\
    &q_{yy}(x,y) - q_{xx}(x,y) = 0,\label{eq:pdekernel_intermediate}\\
    &q_{y}(0) = -i\gamma(x),\quad
    q(x,x) = 0.\label{eq:pdekernelbc2_intermediate}
\end{align}

It can be easily seen that the solution of \eqref{eq:odekernel_intermediate}-\eqref{eq:odekernelbc2_intermediate} is 
\begin{align}
    \gamma(x) = -\overline{K},\qquad \forall x \in[0,1].
\end{align}

Furthermore, the solution of \eqref{eq:pdekernel_intermediate}-\eqref{eq:pdekernelbc2_intermediate} is
\begin{align}
    q(x,y) = i\int_{0}^{x-y}\gamma(\sigma)d\sigma = -i\overline{K}(x-y). \label{eq:solutionkernelpde_intermediate}
\end{align}

Then, from \eqref{eq:backsteppingtransformation1} and \eqref{eq:intermediateBC1}, the control law is given by
\begin{align}
    U(t) = W(t) - i\overline{K}\int_{0}^{1} (1-y)v(t,y)dy - \overline{K}\vartheta_{s}(t), \label{eq:controller_schrodinger}
\end{align}

\textbf{Second backstepping transformation:} Now, consider the following backstepping transformation:
\begin{align}
    w(t,x) = z(t,x) - \int_{0}^{x}\kappa(x,y)z(t,y)dy, \label{eq:backsteppingtransformation2}
\end{align}
\noindent to transform \eqref{eq:intermediateODE}-\eqref{eq:intermediateBC1} into \eqref{eq:target_ode}-\eqref{eq:target_bc}.

Differentiating \eqref{eq:backsteppingtransformation2} once with respect to time and twice with respect to space, substituting \eqref{eq:intermediateODE}-\eqref{eq:intermediateBC1} into it, and plugging the expressions into \eqref{eq:target_ode}-\eqref{eq:target_bc}, we obtain that \eqref{eq:intermediateODE}-\eqref{eq:intermediateBC1} is mapped into \eqref{eq:target_ode}-\eqref{eq:target_bc} if, and only if, the kernel $\kappa$ satisfies the following PDE:
\begin{align}
    &\kappa_{xx}(x,y)-\kappa_{yy}(x,y) = ic\kappa (x,y),\label{eq:kernelpdefinal}\\
    &\kappa_{y}(x,0) = 0, \quad
    \kappa(x,x) = -i\frac{c}{2}x.\label{eq:kernelbc2final}
\end{align}

The solution to the PDE \eqref{eq:kernelpdefinal}-\eqref{eq:kernelbc2final} is given in page 66 of \cite{Krsti2008BoundaryCO}, by
\begin{align}
    \kappa(x,y) = \kappa_{r}(x,y) + i\kappa_{i}(x,y),
\end{align}
\noindent where 
\begin{align*}
    \kappa_{r}(x,y) &= x\sqrt{\frac{c}{2(x^2-y^2)}}\left[-\mbox{ber}_{1}\biggl(\sqrt{c(x^2-y^2)}\biggl) \right.\\
    & \left. - \mbox{bei}_{1}\biggl(\sqrt{c(x^2-y^2)}\biggl)\right],\\
    \kappa_{i}(x,y) &= x\sqrt{\frac{c}{2(x^2-y^2)}}\left[\mbox{ber}_{1}\biggl(\sqrt{c(x^2-y^2)}\biggl) \right.\\
    & - \left. \mbox{bei}_{1}\biggl(\sqrt{c(x^2-y^2)}\biggl) \right],
\end{align*}
and ber$_{1}$ and bei$_{1}$ are the Kelvin functions.

Then, evaluating \eqref{eq:backsteppingtransformation2} at $x=1$ and using \eqref{eq:target_bc}, we obtain 
\begin{align}
    W(t) = \int_{0}^{1}\kappa(1,y)z(t,y)dy. \label{eq:intermediatecontroller}
\end{align}

Plugging \eqref{eq:intermediatecontroller} into \eqref{eq:controller_schrodinger}, with the help of \eqref{eq:backsteppingtransformation1}, we obtain the backstepping feedback control
\begin{align}
    U(t) &= \int_{0}^{1}\left[\kappa_{r}+\int_{y}^{1}\kappa_{i}(1,\xi)(\xi-y)d\xi\right]v(t,y)dy\nonumber\\
    &-\overline{K}\left[1 - \int_{0}^{1}\kappa_{r}(1,y)dy \right]\vartheta_{s} (t) \nonumber\\
    &+ i\left[\int_{0}^{1}\biggl[\kappa_{i}(1,y)-\overline{K}(1-y)\right.\nonumber\\
    &\left.-\int_{y}^{1}\kappa_{r}(1,\xi)(\xi-y)d\xi\right]v(t,y)dy \nonumber\\
    &\left.+\overline{K}\int_{0}^{1}\kappa_{i}(1,y)dy\vartheta_{s}(t)\right]. \label{eq:control_law_mean}
\end{align}

\subsubsection{Invertibility of the Transformations}
As demonstrated in \cite{Ren2013}, the transformations \eqref{eq:backsteppingtransformation1} and \eqref{eq:backsteppingtransformation2} are invertible. Specifically, by postulating the inverse transformation of \eqref{eq:backsteppingtransformation1} as
\begin{align*}
    v(t,x) = z(t,x) - \int_{0}^{x} \iota(x,y) z(t,y) \, dy - \psi(x) \vartheta_{s}(t),
\end{align*}
\noindent and similarly, for \eqref{eq:backsteppingtransformation2},
\begin{align*}
    z(t,x) = w(t,x) - \int_{0}^{x} \eta(x,y) w(t,y) \, dy - \chi(x) \vartheta_{s}(t),
\end{align*}
one can determine the kernels $\iota$, $\eta$, $\psi$, and $\chi$ using the same reasoning as in the direct transformation. As a result, the closed-loop system and the target system exhibit identical stability properties.

\subsection{Target System of the Euler-Bernoulli Beam Equation}
To find out what the actual target system of the EB beam PDE \eqref{eq:ODE1errordynamics}-\eqref{eq:BCUerrordynamics} by using the transformations \eqref{eq:backsteppingtransformation1} and \eqref{eq:backsteppingtransformation2} and control law \eqref{eq:control_law_mean}, let us define 
\begin{align}
    \zeta(t,x) = \int_{x}^{1}\int_{0}^{y}\mbox{Im}\{w(t,\xi)\}d\xi dy,\label{eq:transformation_bernoulli}
\end{align}
\noindent where $w$ is the state of the target system \eqref{eq:target_pde}-\eqref{eq:target_bc} for the Schr\"{o}dinger equation. Computing the second and fourth-order partial derivatives of \eqref{eq:transformation_bernoulli} in time and space, integrating by parts, and using the boundary conditions in \eqref{eq:target_bc}, we verify that $\zeta$ satisfies the following PDE: 
\begin{align}
    &\zeta_{tt}(t,x) + 2c \zeta_{t}(t,x) + c^2 \zeta + \zeta_{xxxx}(t,x) = 0, \label{eq:target_bernoulli_pde}\\
    &\zeta_{x}(t,0) = 0,  \qquad\zeta_{xxx}(t,0) = 0,\label{eq:target_bernoulli_bc2}\\
    &\zeta(t,1) = 0,  \qquad \;\;\;\zeta_{xx}(t,1) = 0.\label{eq:target_bernoulli_bc4}
\end{align}

Additionally, by noticing from \eqref{eq:transformation_bernoulli} and \eqref{eq:target_pde}-\eqref{eq:target_bc}, that the state $w$ is expressed through $\zeta$ as 
\begin{align}
    w(t,x) = \zeta_{t}(t,x) + c\zeta(t,x) - i\zeta_{xx}(t,x),\label{eq:relation_w_zeta}
\end{align}
\noindent it follows that
\begin{align}
    \dot{\vartheta}_{s}(t) = -\overline{K}\vartheta_{s}(t) + \zeta_{t}(t,0) + c\zeta(t,0) - i\zeta_{xx}(t,0). \label{eq:ISS}
\end{align}

To establish the stability for this system, let us consider the state space  $\mathcal{H}_{\zeta} = \mathbb{C}\times\mathcal{H}^{2}_{L}(0,1)\times L^{2}(0,1)$, where $$\mathcal{H}_{L}^{2}=\{f\in H^{2}(0,1)|\; f'(0)=f(1)=0\},$$ and the following induced norm is used $$\|(X,\;f,\; g)\|^{2}_{\mathcal{H}_{\zeta}}=|X|^{2}+\int_{0}^{1}\biggl( |f''(x)|^{2}+|g(x)|^{2}\biggl)dx.$$

Furthermore, define the following operators for \eqref{eq:target_bernoulli_pde}-\eqref{eq:ISS}:
\begin{align}
    \mathcal{A}_{\zeta_{1}}(X, f, g) = \left(-\overline{K}X + g(0)+cf(0)-if''(0),\;g,\;- f^{(4)} \right),\label{eq:operator1bernoulli}
\end{align}
\noindent $\forall (X, f,g)\in D(\mathcal{A}_{\zeta_{1}})$, with 
\begin{align}
D(\mathcal{A}_{\zeta_{1}}) = \left\{ (X, \; f,\;g )\in \mathcal{H}_{\zeta}|\; X\in \mathbb{C},\; \zeta \in H^{4}(0,1), \right. \nonumber\\
\left. g\in \mathcal{H}_{L}^{2}, f'''(0)=f''(1) = 0 \right\},
\end{align}
and
\begin{align}
   \mathcal{A}_{\zeta_{2}}(X, \;f , \; g) = \left( 0\;, 0,\; -2cg-c^2f \right),\label{eq:operator2bernoulli}
\end{align}
\noindent $\forall (X,\; f,\; g)\in \mathcal{H}_{\zeta}$.

Then, system \eqref{eq:target_bernoulli_bc2}-\eqref{eq:ISS} can be written as
\begin{align}
    \frac{dY_{\zeta}(t)}{dt} &= \left(\mathcal{A}_{\zeta_{1}} + \mathcal{A}_{\zeta_{2}} \right)Y_{\zeta}(t),\label{eq:bernoulli_target_operator}\\
    Y_{\zeta}(0) &= Y_{\zeta0}. \label{eq:bernoulli_target_initial_condition}
\end{align}
\noindent where $Y_{\zeta} = (X,\;f,\; g)$.

Using these results, the exponential stability of \eqref{eq:bernoulli_target_operator}-\eqref{eq:bernoulli_target_initial_condition} is formally stated as follows \cite{Smyshlyaev2009}.

\begin{lem}\label{lem:properties_operator_target_bernoulli}
    Let $\mathcal{A}_{\zeta_{1}}$ and $\mathcal{A}_{\zeta_{2}}$ be defined by \eqref{eq:operator1bernoulli}-\eqref{eq:operator2bernoulli}, respectively. Then
    \begin{itemize}
        \item $\mathcal{A}_{\zeta_{1}}^{-1}$ exists and is compact on $\mathcal{H}_{\zeta}$ and hence the spectrum of $\mathcal{A}_{\zeta_{1}}$ consists of isolated eigenvalues of finitely algebraic multiplicity only, which are given by
        \begin{align*}
            \sigma_{0}&=-\overline{K}, &\sigma_{n} =& -c \pm i \frac{\pi^{2}(2n+1)^{2}}{4}, \qquad m\in\mathbb{N}.
        \end{align*}
        \item There is a sequence of eigenfunctions of $\mathcal{A}_{\zeta_{1}}$ which forms a Riesz basis for $\mathcal{H}_{\zeta_{1}}$.
        \item $\left(\mathcal{A}_{\zeta_{1}} + \mathcal{A}_{\zeta_{2}} \right)$ generates an exponentially stable C$_{0}$-semigroup $\mathrm{e}^{\left(\mathcal{A}_{\zeta_{1}} + \mathcal{A}_{\zeta_{2}} \right)t}$ in the sense $$\|\mathrm{e}^{\left(\mathcal{A}_{\zeta_{1}} + \mathcal{A}_{\zeta_{2}} \right) t}\|_{\mathcal{H}_{\zeta}}\leq M_{2}\mathrm{e}^{-ct}, \quad M_{2}>0.$$ 
    \end{itemize}
\end{lem}


\subsection{Control Laws}
In this section, the control law \eqref{eq:control_law_filtered} will be rewritten in terms of the EB beam PDE states (see \eqref{eq:ODE1errordynamics}-\eqref{eq:BCUerrordynamics}). Using \eqref{eq:backsteppingtransformation1} and \eqref{eq:backsteppingtransformation2}, we have
\begin{align}
    w(t,x) &= v(t,x) + i\overline{K}\int_{0}^{x}(x-y)v(t,y)dy \nonumber\\
    &-\int_{0}^{x}\kappa(x,y)\biggl[v(t,y) +i\overline{K}\int_{0}^{y}(y-\xi)v(t,\xi)d\xi\biggl]dy\nonumber\\ &-\overline{K}\biggl[\int_{0}^{x}\kappa(x,y)dy - 1\biggl]\vartheta_{s}(t),
\end{align}
\noindent and using \eqref{eq:transformation_schrodinger} and \eqref{eq:relation_w_zeta}, it follows that the transformation \eqref{eq:backsteppingtransformation2} becomes
\begin{align}
    \zeta_{t}(t,x) &+ c\zeta(t,x) = \beta_{t}(t,x) +p_{1}(x)\vartheta_{s}(t) \nonumber\\
    &\hspace{-0.5cm}- \int_{0}^{x}\biggl[\kappa_{r}(x,y)+f_{1}(x,y)\biggl] \beta_{t}(t,y)dy \nonumber\\
    &\hspace{-0.5cm}-\int_{0}^{x}\biggl[\kappa_{i}(x,y) - \overline{K}(x-y) + f_{2}(x,y)\biggl]\beta_{xx}(t,y)dy, \label{eq:transformation_bernoulli1} \\
    \zeta_{xx}(t,x) &= \beta_{xx}(t,x) + p_{2}(x)\vartheta_{s}(t) \nonumber\\
    & - \int_{0}^{x}\biggl[\overline{K}(x-y)-\kappa_{i}(x,y) - g_{1}(x,y)\biggl]\beta_{t}(t,y)dy\nonumber\\
    & -\int_{0}^{x}\biggl[\kappa_{r}(x,y) - g_{2}(x,y)\biggl]\beta_{xx}(t,y)dy,\label{eq:transformation_bernoulli2}
\end{align}
\noindent where
\begin{align*}
    f_{1}(x,y) &= \overline{K}\int_{y}^{x}\kappa_{i}(x,\xi)(\xi-y)d\xi,\\
    f_{2}(x,y) &= \overline{K}\int_{y}^{x}\kappa_{r}(x,\xi)(\xi-y)d\xi,\\
    g_{1}(x,y) &= \overline{K}\int_{y}^{x}\kappa_{r}(x,\xi)(\xi-y)d\xi,\\
    g_{2}(x,y) &=\overline{K}\int_{y}^{x}\kappa_{i}(x,\xi)(\xi-y)d\xi,\\
    p_{1}(x) &= \overline{K}\left(1-\int_{0}^{x}\left[\kappa_{r}(x,y)-\kappa_{i}(x,y)\right]dy\right),\\
    p_{2}(x) &= \overline{K}\int_{0}^{x}\left[\kappa_{i}(x,y)+\kappa_{r}(x,y)\right]dy.
\end{align*}

The controls are obtained by setting $x=1$ in \eqref{eq:transformation_bernoulli1}-\eqref{eq:transformation_bernoulli2}:
\begin{align}
    \dot{U}_{1}(t) &= \int_{0}^{1}\biggl[\kappa_{r}(1,y) - f_{1}(1,y)\biggl] \beta_{t}(t,y)dy \nonumber\\
    & + \int_{0}^{1}\biggl[ \kappa_{i}(1,y) - \overline{K}(x-y) + f_{2}(1,y)\biggl]\beta_{xx}(t,y)dy\nonumber\\
    & - p_{1}(1)\vartheta_{s}(t), \label{eq:expression1_controlU1}\\
    U_{2}(t) &=  \int_{0}^{1}\biggl[\overline{K}(1-y)-\kappa_{i}(1,y) - g_{1}(1,y)\biggl]\beta_{t}(t,y)dy \nonumber \\
    & +\int_{0}^{1}\biggl[\kappa_{r}(1,y) - g_{2}(1,y)\biggl]\beta_{xx}(t,y)dy \nonumber\\
    & - p_{2}(1)\vartheta_{s}(t). \label{eq:expression1_controlU2}
\end{align}

Importantly, the control law \eqref{eq:expression1_controlU1} must be implemented as integral due to the boundary condition \eqref{eq:BCUerrordynamics}. Another observation we make is that even though the states $\vartheta_{s}$ and $\zeta$ converge exponentially to zero, the same cannot be said about $\beta$ and $\vartheta$. Indeed, when $v$ converges to zero, $\beta$ may converge to an arbitrary constant due to the equality $v(t,x)=\beta_{t}(t,x) - i\beta_{xx}(t,x)$.

\subsubsection{Achieving Regulation to zero}
In order to achieve regulation to zero, we are going to modify the control law \eqref{eq:expression1_controlU1}. Our objective is to express $\beta_{xx}$ in \eqref{eq:expression1_controlU1} through the time derivatives $\beta_{t}$ and $\beta_{tt}$.

Twice integrating the PDE \eqref{eq:PDEerrordynamics} with respect to $x$, first from $0$ to $x$, and them from $x$ to $1$, results in
\begin{align}
    \beta_{xx}(t,x) = \beta_{xx}(t,1) + \int_{x}^{1}\int_{0}^{y}\beta_{tt}(t,\xi)d\xi dy.\label{eq:bernoulli_twiceintegrated}
\end{align}

Substituting \eqref{eq:expression1_controlU2} into \eqref{eq:bernoulli_twiceintegrated}, we get
\begin{align}
    \beta_{xx}(t,x) &= \int_{0}^{1}\biggl[\overline{K}(1-y)-\kappa_{i}(1,y) - g_{1}(1,y)\biggl]\beta_{t}(t,y)dy\nonumber \\
     & +\int_{0}^{1}\biggl[ \kappa_{r}(1,y) - g_{2}(1,y)\biggl]\beta_{xx}(t,y)dy\nonumber\\
     & + \int_{x}^{1}\int_{0}^{y}\beta_{tt}(t,\xi)d\xi dy- p_{2}(1)\vartheta_{s}(t).\label{eq:expression1_betaxx}
\end{align}

In order to make progress, we will introduce the following notation
\begin{align}
    \mathcal{F}_{1}(x,y) &= \kappa_{r}(x,y)-f_{1}(x,y),\label{eq:operato1}\\
    \mathcal{F}_{2}(x,y) &= \kappa_{i}(x,y) - \overline{K}(x-y) + f_{2}(x,y),\label{eq:operato2}\\
    \mathcal{R}_{1}(x,y) &= \overline{K}(x-y)-\kappa_{i}(x,y) - g_{1}(x,y),\label{eq:operato3}\\
    \mathcal{R}_{2}(x,y) &= \kappa_{r}(x,y) - g_{2}(x,y).\label{eq:operato4}    
\end{align}

Then, multiplying \eqref{eq:expression1_betaxx} by $\mathcal{R}_{2}(1,y)$ and integrating from $0$ to $1$, yields
\begin{align}
    &\int_{0}^{1}\mathcal{R}_{2}(1,y)\beta_{xx}(t,y)dy =  \int_{0}^{1}\mathcal{R}_{2}(1,y)dy\nonumber\\
    &\hspace{0.1cm}\times\left[\int_{0}^{1}\mathcal{R}_{1}(1,y)\beta_{t}(t,y)dy + \int_{0}^{1}\mathcal{R}_{2}(1,y)\beta_{xx}(t,y)dy\right]\nonumber\\
    &\hspace{0.1cm}+ \int_{0}^{1}\mathcal{R}_{2}(1,y)\int_{y}^{1}\int_{0}^{z}\beta_{tt}(t,\xi)d\xi dzdy\nonumber\\
    &\hspace{0.1cm}-p_{2}(1)\vartheta_{s}(t) \int_{0}^{1}\mathcal{R}_{2}(1,y)dy .
\end{align}

Therefore,
\begin{align}
    &\int_{0}^{1}\mathcal{R}_{2}(1,y)\beta_{xx}(t,y)dy =  \frac{\varphi_{r}-1}{\varphi_{r}}\int_{0}^{1}\mathcal{R}_{1}(1,y)\beta_{t}(t,y)dy\nonumber\\
    & \hspace{1.4cm}-\frac{1}{\varphi_{r}}\int_{0}^{1}\biggl(\mathcal{Q}(1,y) - (1-\varphi_{r})(1-y)\biggl)\beta_{tt}(t,y)dy\nonumber\\
    &\hspace{1.4cm} - \frac{p_{2}(1)}{\varphi_{r}}\vartheta_{s}(t)\int_{0}^{1}\mathcal{R}_{2}(1,y)dy,    \label{eq:expressionBetaxx}
\end{align}

\noindent where
\begin{align*}
    \mathcal{Q}(x,y) &= \int_{y}^{x}\mathcal{R}_{2}(x,\xi)(\xi-y)d\xi,\\
    \varphi_{r} &= 1- \int_{0}^{1}\mathcal{R}_{2}(1,y)dy.
\end{align*}

Substituting \eqref{eq:expressionBetaxx} into \eqref{eq:expression1_betaxx}, considering \eqref{eq:operato1}-\eqref{eq:operato4}, one has
\begin{align}
    &\beta_{xx}(t,x) = -\frac{1}{\varphi_{r}}\int_{0}^{1}\mathcal{R}_{1}(1,y)\beta_{t}(t,y)dy\nonumber \\
    &-\frac{1}{\varphi_{r}}\int_{0}^{1}\left[\mathcal{Q}(1,y) - (1-\varphi_{r})(1-y)\right]\beta_{tt}(t,y)dy\nonumber\\
    &+ \int_{x}^{1}\int_{0}^{y}\beta_{tt}(t,\xi)d\xi - p_{2}(1)\frac{(1+\varphi_{r})}{\varphi_{r}}\vartheta_{s}(t)\int_{0}^{1}\mathcal{R}_{2}(1,y)dy.\label{eq:expression2_betaxx}
\end{align}

Then, substituting \eqref{eq:expression2_betaxx} into \eqref{eq:expression1_controlU1}:
\begin{align}
    \dot{U}_{1}(t) &= \int_{0}^{1} \biggl(\mathcal{F}_{1}(1,y) - \frac{\mathcal{R}_{1}(1,y)}{\varphi_{r}} \mathcal{F}_{2}(1,y)\biggl) \beta_{t}(t,y)dy\nonumber\\
    & - \int_{0}^{1}\left[\mathcal{S}(1,y) + \frac{\varphi(1)}{\varphi_{r}}(1-y-\mathcal{Q}(1,y))\right]\beta_{tt}(t,y)dy\nonumber\\
    & -\left[p_{1}(1) +p_{2}(1)\frac{\varphi(1)}{\varphi_{r}}\int_{0}^{1}\mathcal{R}_{2}(1,y)dy\right]\vartheta_{s}(t).\label{eq:integral_control}
\end{align}
\noindent where $\mathcal{S}(x,y) = \int_{y}^{x}\mathcal{R}_{2}(\xi-y)d\xi$ and $\varphi(x) = -\int_{0}^{x}\mathcal{R}_{2}(x,y)dy$. 
Integrating \eqref{eq:integral_control} with respect to time and using \eqref{eq:hessian_estimates}, we finally get the controller
\begin{align}
    U_{1}(t) &= \int_{0}^{1} \biggl(\mathcal{F}_{1}(1,y) + \frac{\mathcal{R}_{1} (1)}{\varphi_{r}} \mathcal{F}_{2}(1,y)\biggl)\beta(t,y)dy \nonumber\\
    & - \int_{0}^{1}\biggl[\mathcal{S}(1,y) + \frac{\varphi(1)}{\varphi_{r}}(1-y-\mathcal{Q}(1,y))\biggl]\beta_{t}(t,y)dy \nonumber\\
    & -\biggl[p_{1}(1) +p_{2}(1)\frac{\varphi(1)}{\varphi_{r}}\int_{0}^{1}\mathcal{R}_{2}(1,y)dy\biggl]\vartheta(t).
\end{align}

The controller $U_{2}$ can also be obtained in a similar manner:
\begin{align}
    U_{2}(t) &= \frac{c^{2}}{8}\beta(t,1)+\int_{0}^{1}\mathcal{F}_{1}(1,y)\beta (t,y)dy \nonumber\\
    &-\int_{0}^{1}\mathcal{F}_{2}(1,y)\beta_{t}(t,y)dy - p_{2}(1)\vartheta(t). \label{eq:control_law2}
\end{align}

\subsubsection{Implementable Extremum Seeking Control Law}
Introducing a result of \cite{Hale1990}, the averaged version of the gradient and Hessian estimate are calculated as
\begin{align}
    G_{av} (t) = H\vartheta_{av}(t), \quad \hat{H}_{av}(t) = H.\label{eq:hessian_estimates}
\end{align}

From \eqref{eq:control_schrodinger} and \eqref{eq:control_law_mean}, choosing $\overline{K}=KH$ with $K>0$, plugging the average gradient and Hessian estimates \eqref{eq:hessian_estimates}, and  introducing a low-pass filter with cut frequency $\overline{c}$ , we obtain the non-average controllers
\small{\begin{align}
    U_{1}(t) &= \frac{\overline{c}}{s+\overline{c}}\left\{  H\int_{0}^{1} \left[]\mathcal{F}_{1}(1,y) + \frac{\mathcal{R}_{1} (1,y)}{\varphi_{r}} \mathcal{F}_{2}(1,y)\right]\beta(t,y)dy \right.\nonumber\\
    &- H\int_{0}^{1}\biggl[\mathcal{S}(1,y) + \frac{\varphi(1)}{\varphi_{r}}(1-y-\mathcal{Q}(1,y))\biggl]\beta_{t}(t,y)dy\nonumber\\
    &\left. -\left[p_{1}(1) +p_{2}\frac{\varphi(1)}{\varphi_{r}}\int_{0}^{1}\mathcal{R}_{2}(1,y)dy\right]G(t)\right\},\label{eq:control_law_filtered}\\
    U_{2}(t)&=  \frac{\overline{c}}{s+\overline{c}}\left\{H\frac{c^{2}}{8}\beta(t,1) + \int_{0}^{1}\mathcal{F}_{1}(1,y)\beta (t,y)dy\right.\nonumber \\
    &\left.- H\int_{0}^{1}\mathcal{F}_{2}(1,y)\beta_{t}(t,y)dy - p_{2}(1)G(t)\right\}. \label{eq:control_law_filtered_2}
\end{align}}

\normalsize
\section{Stability Analysis}\label{section:stability}
In this section, the well-posedness and exponential stability of the proposed ES methodology is proved. First, define the state space
\begin{multline*}
    \hspace{-0.3cm} \mathcal{H}_{c} = \biggl\{ (X,\;f,\; g)\in \mathbb{R}\times H^{2}(0,1)\times L^{2}(0,1)\mid\,f'(0)=0,\\
    f(1) = \int_{0}^{1}\left(\mathcal{F}_{1}(1,y) + \frac{\varphi (1)}{\varphi_{r}} \mathcal{F}_{2}(1,y)\right)f(y)dy \\    
     - \int_{0}^{1}\left[\mathcal{S}(1,y) + \frac{\varphi(1)}{\varphi_{r}}(1-y-\mathcal{Q}(1,y))\right]g(y)dy \\
    \hspace{2cm} \left. -\left[p_{1}(1) +p_{2}(1)\frac{\varphi(1)}{\varphi_{r}}\int_{0}^{1}\mathcal{R}_{2}(1,y)dy\right]X(t) \right\},
\end{multline*}

\noindent with the following inner product induced norm of $\mathcal{H}_{c}$:
\begin{align*}
    \| (X,\;f,\;g) \|^{2}_{\mathcal{H}_{c}} = |X|^{2}+ \int_{0}^{1}\biggl(|f''(x)|^{2}+|g(x)|^{2}\biggl)dx.
\end{align*}

The closed-loop system can be written as
\begin{align}
    \frac{dY_{c}(t)}{dt} = \mathcal{A}Y_{c}(t),\label{eq:closedloop_operator}
\end{align}
\noindent where $Y_{{c}}=(X, \; f,\; g)$
\begin{align*}
   \hspace{-0.2cm} \mathcal{A}(X,f,g) &= \left(-\overline{K}X + g(0)+cf(0)-if''(0),g,-f^{(4)}\right),
    \end{align*}
\noindent $\forall (X,f,g)\in D(\mathcal{A})$ and
\begin{multline*}
    \hspace{-0.3cm} D(\mathcal{A}) = \bigg\{ (X,f,g)\in \mathcal{H}_{c}\mid \mathcal{A}\in \mathcal{H}_{c},\, f'''(0)=0, \\ 
      f''(1) = \frac{c^{2}}{8}f(1) + \int_{0}^{1}\biggl[\mathcal{F}_{1}(1,y)f(y) \\
      \hspace{3.8cm} -\mathcal{F}_{2}(1,y)g(y)\biggl]dy - p_{2}(1)X(t) \bigg\}.
\end{multline*}

The existence and boundedness of $\mathcal{A}^{-1}$, as well as the existence and uniqueness of a classical solution to \eqref{eq:closedloop_operator}, were established in \cite{Smyshlyaev2009}. With this in mind, we now present the main result of this paper.

\begin{thm}
    Consider the control system in Figure \ref{fig:euler}, with control laws $U_{1}$ and $U_{2}$ given in \eqref{eq:control_law_filtered}-\eqref{eq:control_law_filtered_2}, respectively. There exists $\overline{c}^{*}>0$ such that, $\forall \overline{c}\geq \overline{c}^{*}$, $\exists \omega^{*}(\overline{c})>0$ such that, $\forall \omega\geq \omega^{*}$, and $K>0$ sufficiently large, the closed-loop system \eqref{eq:ODE1errordynamics}-\eqref{eq:BCUerrordynamics} has a unique locally exponentially stable periodic solution in $t$ with a period $\Pi := 2\pi/\omega$, denoted as $\vartheta^{\Pi}(t)$, $\beta^{\Pi}(t,x)$. This solution satisfies the condition
    \begin{align}
        \|(\vartheta(t),\;\zeta_{xx}(t),\;\zeta_{t}(t))\|_{\mathcal{H}_{c}} \leq \mathcal{O}(1/\omega),\label{eq:convergance_closedloop}
    \end{align}
   
    Furthermore
    \begin{align}
        \lim_{t\rightarrow\infty}\sup |\Theta(t)-\Theta^{*}|&=\mathcal{O}(|a|+1/\omega), \label{eq:convergence_Theta}\\
        \lim_{t\rightarrow\infty}\sup |\theta_{1}(t)-\Theta^{*}|&=\mathcal{O}(r(\omega) +1/\omega), \label{eq:convergence_theta1}\\
        \lim_{t\rightarrow\infty}\sup|y(t)-y^{*}| &= \mathcal{O}(|a|^{2}+1/\omega^2), \label{eq:convergence_y}
    \end{align}
    \noindent where $r(\omega) = \frac{|a(\cosh(\sqrt{w})+\cos(\sqrt{w}))|}{2}$.    
\end{thm}

\begin{proof}
First, note that the eigenvalues of $\mathcal{A}$ are
\begin{align*}
    \sigma_{0} = -\overline{K}, \qquad \sigma_{n} = -c \pm i \frac{\pi^{2}(2n+1)^{2}}{4},
\end{align*}
\noindent where $n\in \mathbb{R}$.

From this, and the existence and boundedness of $\mathcal{A}^{-1}$, and Theorem 1.3 of \cite{pazy2012}, it follows that $\mathcal{A}$ generates a C$_{0}$-semigroup on $\mathcal{H}_{c}$. Therefore, for any initial value $(\vartheta(0),\,\zeta(0),\zeta_{t}(0))\in \mathcal{H}_{c}$, there exists a unique solution to \eqref{eq:closedloop_operator}.

By the density of $D(\mathcal{A})$ in $\mathcal{H}_{c}$, and the inverse backstepping transformations \eqref{eq:backsteppingtransformation1} and \eqref{eq:backsteppingtransformation2}, and Lemma \ref{lem:properties_operator_target_bernoulli}, it follows that for any $\varepsilon>0$, there exists $M_{\varepsilon}>0$, such that for all initial conditions $(\vartheta (0),\;\zeta(0),\;\zeta_{t}(0))\in \mathcal{H}_{c}$,
\begin{multline*}
    \left\|(\vartheta(t),\;\zeta_{xx}(t),\;\zeta_{t}(t))\right\|_{\mathcal{H}_{c}} \leq \\M_{\varepsilon}\mathrm{e}^{(-c+\varepsilon)t}\left\|(\vartheta(0),\;\zeta_{xx}(0),\;\zeta_{t}(0))\right\|_{\mathcal{H}_{c}}.
\end{multline*}

Then, according to the averaging theory in infinite dimensions \cite{Hale1990}, for $\omega$ sufficiently large, the closed-loop system \eqref{eq:ODE1errordynamics}-\eqref{eq:BCUerrordynamics}, with $U_{1}$ and $U_{2}$ defined in \eqref{eq:control_law_filtered} and \eqref{eq:control_law_filtered_2}, respectively, has a unique exponentially stable periodic solution around its equilibrium satisfying \eqref{eq:convergance_closedloop}.

The asymptotic convergence to a neighborhood of the extremum point is proved taking the absolute value of the second expression in \eqref{eq:error_sigma} after replacing $\hat{\Theta} = \vartheta + \Theta^{*}$ from \eqref{eq:error_vartheta}, resulting in $        |\Theta(t)-\Theta^{*}| = |\vartheta(t) + a\sin (\omega t)|$. From this, and writing it by adding and subtracting the periodic solution $\vartheta^{\Pi}$, it follows that
    \begin{align}
        |\Theta(t)-\Theta^{*}| = |\vartheta(t) -\vartheta^{\Pi}(t) + \vartheta^{\Pi}(t)+ a\sin (\omega t)|.\label{eq:abs_Thetaerror}
    \end{align}

By applying the average theorem, one can conclude that $\vartheta(t) -\vartheta^{\Pi}(t)\rightarrow 0$ as $t\rightarrow\infty$. Consequently,
\begin{align}
        \lim_{t\rightarrow\infty}\sup |\Theta(t)-\Theta^{*}| = \lim_{t\rightarrow\infty}\sup | \vartheta^{\Pi}(t)+ a\sin (\omega t)|.
\end{align}

Finally, using the relationship \eqref{eq:convergance_closedloop}, we get the result presented in \eqref{eq:convergence_Theta}.

Since $\theta_{1}(t)-\Theta^{*}=\tilde{\theta}_{1}+S_{1}(t)$ from \eqref{eq:error_sigma}-\eqref{eq:error_vartheta}, and recalling that $S_{1}$ is of order $\mathcal{O}\left( \frac{|a(\cosh(\sqrt{w})+\cos(\sqrt{w}))|}{2}\right)$, we get the ultimate bound in \eqref{eq:convergence_theta1}.

In order to show the convergence of the output $y$, we can follow the same steps employed for $\Theta$ by plugging \eqref{eq:abs_Thetaerror} into \eqref{extrachapter.eq:final_output_static_map}, such that
\begin{align*}
        \lim_{t\rightarrow\infty}\sup|y(t)-y^{*}| = \lim_{t\rightarrow\infty}\sup |H\vartheta^{2}(t)+Ha^{2}\sin(\omega t)^{2}|.
\end{align*}

Hence, by rewriting the above equation in terms of $\vartheta^{\Pi}$ and again with the help of \eqref{eq:convergance_closedloop}, we finally get \eqref{eq:convergence_y}. The proof is complete.
\end{proof}

\section{Simulation results}\label{section:simulation}
The numerical implementation of the Euler-Bernoulli equation was carried out using the finite element method with cubic Hermitian functions. Numerical simulations illustrate the stability and convergence properties of the proposed ES scheme, where the actuation dynamics are governed by the one-dimensional Euler-Bernoulli PDE.

Considering a quadratic static map as in \eqref{extrachapter.eq:static_map}, the system is subjected to the control laws \eqref{eq:control_law_filtered} and \eqref{eq:control_law_filtered_2}. The Hessian is given by $H=-1$, with an optimizer $\Theta^{*}=1.5$ and an optimal unknown output value $y^{*}=2.4$. The controller parameters are chosen as $\omega = 5$, $a = 0.2$, $c = 0.1$, $\overline{c}=6$, and $K = 0.1$.

The closed-loop simulation results, presented in Figure \ref{fig:results} and \ref{fig:results3D}, demonstrate the effectiveness of the proposed control approach. The control actions depicted in Figure \ref{fig:results} ensure that the variables $(y,\theta,\Theta)$ converge toward the neighborhood of their optimal values $(y^{*},\Theta^{*},\Theta^{*})$. These findings validate the performance of the ES-based control strategy in driving the system towards optimal operation as shown in Figure \ref{fig:results3D}.

\begin{figure}[htpb]
    \centering
    \includegraphics[width=\linewidth]{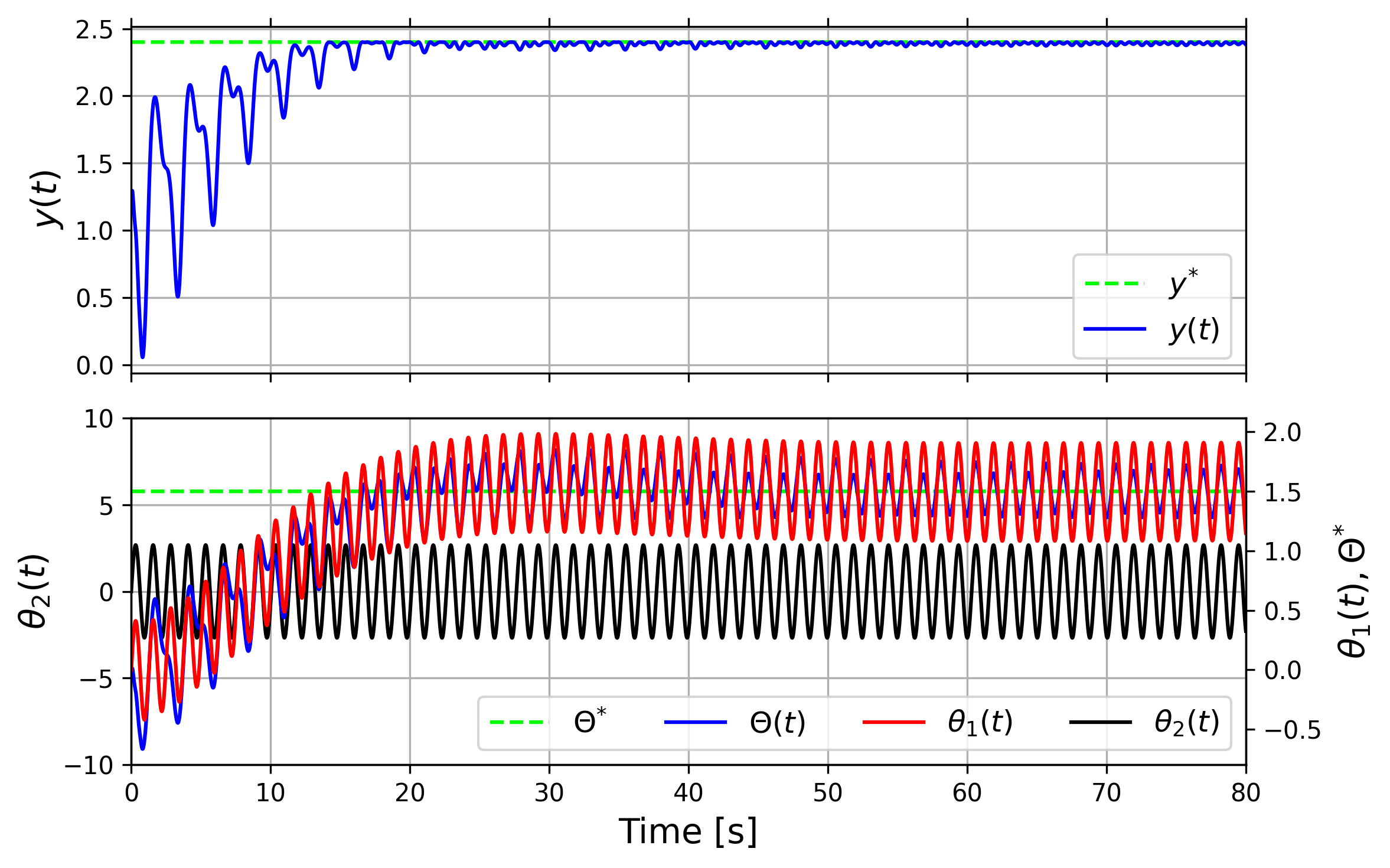}
    \caption{The closed-loop response of the EB beam PDE with ES compensating controller.}
    \label{fig:results}
\end{figure}

\begin{figure}[htpb]
    \centering
    \includegraphics[width=\linewidth]{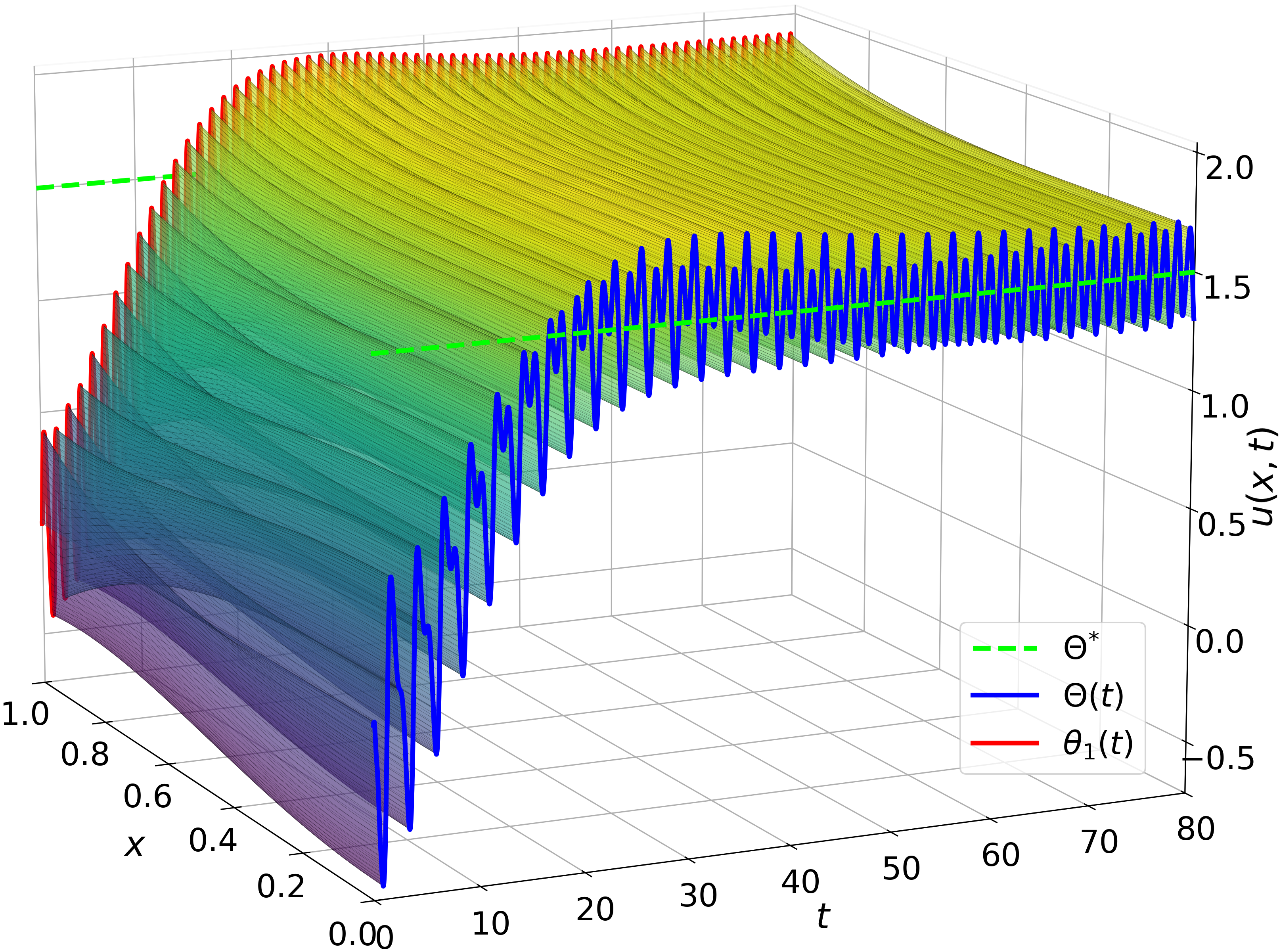}
    \caption{Time evolution of the beam displacement.}
    \label{fig:results3D}
\end{figure}

\section{Conclusions} \label{section:conclusion}
The proposed ES methodology optimizes the quadratic static map by seeking the optimal $\Theta^{*}$ in cascade with EB beam PDEs. While the infinite-dimension actuation dynamics must be known, no prior information about the map parameters is assumed. To compensate for the dynamics, a boundary control law with average-based estimates of the gradient and Hessian of the unknown map is proposed for the EB beam PDE using the backstepping methodology and its corresponding representation using the Schr\"{o}dinger equation. For future work, the approach can be extended to different boundary conditions. However, at this time, it remains uncertain whether a direct connection can be established between the EB equation and the Schr\"{o}dinger equation under distinct boundary conditions. Other possibilities lie in the design and analysis of different control problems with EB beam PDEs, as considered in the following references \cite{paper1,paper2,paper3,paper4,paper5,paper6,paper7,paper8,paper9,paper10,paper11,paper12,paper13,paper14,paper15,paper16,paper17,paper18,paper19,paper20}.

\bibliographystyle{IEEEtran}
\bibliography{references}

\end{document}